\documentclass[12pt]{amsart}

\usepackage{tikz}
\usepackage{comment}

\newtheorem{thm}{Theorem}
\newtheorem{lem}{Lemma}
\newtheorem{corollary}{Corollary}
\newtheorem{obs}{Observation}
\newtheorem{prop}{Proposition}

\DeclareMathOperator{\tp}{top}
\DeclareMathOperator{\bt}{bot}
\DeclareMathOperator{\ar}{area}

\newcommand{\Sos}{\text{S\'os}}

\title{S\'os Permutations}
\author[S. Bockting-Conrad, Y. Kashina, T. K. Petersen,\\and B. E. Tenner]{Sarah Bockting-Conrad, Yevgenia Kashina, T. Kyle Petersen,\\and Bridget Eileen Tenner}
\thanks{Work of Petersen was partially supported by Simons Foundation Collaboration Grant for Mathematicians 353772. Work of Tenner was partially supported by Simons Foundation Collaboration Grant for Mathematicians 277603.}

\begin{document}

\maketitle

\begin{abstract}
Let $f(x) = \alpha x + \beta \mod 1$ for fixed real parameters $\alpha$ and $\beta$. For any positive integer $n$, define the \emph{S\'os permutation} $\pi$ to be the lexicographically first permutation such that $0 \leq f(\pi(0)) \leq f(\pi(1)) \leq \cdots \leq f(\pi(n)) < 1$. In this article we give a bijection between S\'os permutations and regions in a partition of the parameter space $(\alpha,\beta)\in [0,1)^2$. This allows us to enumerate these permutations and to obtain the following ``three areas'' theorem: in any vertical strip $(a/b,c/d)\times [0,1)$, with $(a/b,c/d)$ a Farey interval, there are at most three distinct areas of regions, and one of these areas is the sum of the other two.
\end{abstract}

\section{Introduction}

Let $y = \alpha x + \beta$ be a line in the plane. Let $f=f_{\alpha,\beta}$ be the function that returns the fractional part of a $y$-value on this line, i.e.,
\[
 f(x) = \alpha x + \beta \mod 1 .
\]
Note that because of the reduction to fractional parts, we can assume $\alpha, \beta \in [0,1)$.

Now fix a positive integer $n$, and consider the list $(f(0),f(1),\ldots,f(n))$, which consists of some numbers between $0$ and $1$. We can put these numbers in order by sorting them with a permutation $\pi:\{0,1,\ldots,n\} \to \{0,1,\ldots,n\}$ such that
\[
 f(\pi(0)) \leq f(\pi(1)) \leq \cdots \leq f(\pi(n)).
\]
Generically speaking, all inequalities will be strict and such a $\pi$ is uniquely determined. In the case of a tie, we choose $\pi=\pi_{\alpha,\beta}^{(n)}$ to be the lexicographically first such permutation. We call a permutation arising in this way a \emph{S\'os permutation}. We denote the set of S\'os permutations of size $n$ by
\[
 \Sos_n  = \left\{ \pi_{\alpha,\beta}^{(n)} : \alpha, \beta \in [0,1) \right\}.
\]
Our choice of terminology will be explained shortly.

As an example of a S\'os permutation, consider $\alpha = .44$, $\beta = .32$, and $n=7$. Then $f(x) = .44x + .32 \mod 1$ and evaluating $f$ at the integers from $0$ to $7$ we find
{\small
\[
  (f(0), f(1), f(2), f(3), f(4), f(5), f(6), f(7)) = (.32, .76, .20, .64, .08, .52, .96, .40). \]
}
Sorting this list, we find
\begin{equation*}
\mbox{\small $\displaystyle  (.08, .20, .32, .40, .52, .64, .76, .96) = (f(4), f(2), f(0), f(7), f(5), f(3), f(1), f(6)), $}
\end{equation*}
which corresponds to the permutation $\pi = 42075316$. We see this example illustrated in Figure \ref{fig:linex}.

\begin{figure}
\[
\begin{tikzpicture}
\draw (0,6) node[scale=.8] (a)  {
 \begin{tikzpicture}[scale=1.5]
\draw[step=1cm, gray, very thin] (-.1,-.1) grid (7.9,3.9);
\draw[->,very thick] (-.1,0) -- (7.9,0) node[right] {$x$};
\draw[->,very thick] (0,-.1) -- (0,4) node[right] {$y$};
\draw[red, thick] (0,.32)--(7.5,3.62);
\foreach \x in {1,...,7}{
 \draw (\x,0) node[below] {$\x$};
 }
\foreach \x in {1,2,3}{
 \draw (0,\x) node[left] {$\x$};
 }
\draw (0,.32) node[circle, fill=red, draw = none, inner sep=2] {};
\draw (1,.76) node[circle, fill=red, draw = none, inner sep=2] {};
\draw (2,1.20) node[circle, fill=red, draw = none, inner sep=2] {};
\draw (3,1.64) node[circle, fill=red, draw = none, inner sep=2] {};
\draw (4,2.08) node[circle, fill=red, draw = none, inner sep=2] {};
\draw (5,2.52) node[circle, fill=red, draw = none, inner sep=2] {};
\draw (6,2.96) node[circle, fill=red, draw = none, inner sep=2] {};
\draw (7,3.4) node[circle, fill=red, draw = none, inner sep=2] {};
\end{tikzpicture}
};
\draw (0,0)  node[scale=.8] (b) {
\begin{tikzpicture}[xscale=1.5,yscale=4.5]
\draw[step=1cm, gray, very thin] (-.1,-.1) grid (7.9,1.1);
\draw[->,very thick] (-.1,0) -- (7.9,0) node[right] {$x$};
\draw[->,very thick] (0,-.1) -- (0,1.1) node[right] {$y$};
\draw[red, thick] (0,.32)--(1.5454,1);
\draw[red, thick] (1.5454,0)--(3.8182,1);
\draw[red, thick] (3.8182,0)--(6.09091,1);
\draw[red, thick] (6.09091,0)--(7.5,0.62);
\foreach \x in {1,...,7}{
 \draw (\x,0) node[below] {$\x$};
 }
\draw (0,1) node[left] {$1$};
\draw (0,.32) node[circle, fill=red, draw = none, inner sep=2] {};
\draw (1,.76) node[circle, fill=red, draw = none, inner sep=2] {};
\draw (2,.20) node[circle, fill=red, draw = none, inner sep=2] {};
\draw (3,.64) node[circle, fill=red, draw = none, inner sep=2] {};
\draw (4,.08) node[circle, fill=red, draw = none, inner sep=2] {};
\draw (5,.52) node[circle, fill=red, draw = none, inner sep=2] {};
\draw (6,.96) node[circle, fill=red, draw = none, inner sep=2] {};
\draw (7,.4) node[circle, fill=red, draw = none, inner sep=2] {};
\draw[very thick] (7.5,-.1)--(7.5,1.1);
\draw[dashed, thin] (0,.32) -- (7.5,.32) node[circle, fill=red, draw = none, inner sep=2] {} node[right] {$0$};
\draw[dashed, thin] (1,.76) -- (7.5,.76) node[circle, fill=red, draw = none, inner sep=2] {} node[right] {$1$};
\draw[dashed, thin] (2,.20) -- (7.5,.20) node[circle, fill=red, draw = none, inner sep=2] {} node[right] {$2$};
\draw[dashed, thin] (3,.64) -- (7.5,.64) node[circle, fill=red, draw = none, inner sep=2] {} node[right] {$3$};
\draw[dashed, thin] (4,.08) -- (7.5,.08) node[circle, fill=red, draw = none, inner sep=2] {} node[right] {$4$};
\draw[dashed, thin] (5,.52) -- (7.5,.52) node[circle, fill=red, draw = none, inner sep=2] {} node[right] {$5$};
\draw[dashed, thin] (6,.96) -- (7.5,.96) node[circle, fill=red, draw = none, inner sep=2] {} node[right] {$6$};
\draw[dashed, thin] (7,.4)--(7.5,.4) node[circle, fill=red, draw = none, inner sep=2] {} node[right] {$7$};
\end{tikzpicture}
};
\draw[->,thick] (a)--node[midway, right] {$\mod 1$} (b);
\end{tikzpicture}
\]
\caption{The line $y=.44x+.32$ gives rise to the permutation $\pi=42075316$.}\label{fig:linex}
\end{figure}
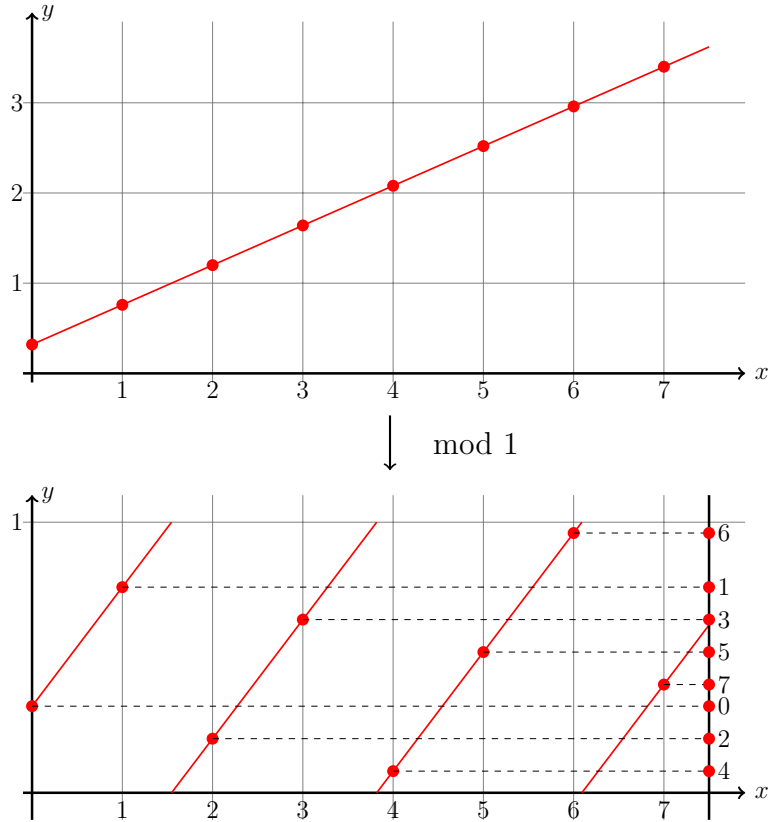

Vertical shifts of our line, that is, increases in $\beta$, cannot change the resulting permutation very much. We clarify our meaning with an example.  Imagine modifying the example discussed above by increasing from $\beta =.32$ to $\beta=.36$. Before, our largest entry was $f_{.44,.32}(6)=.96$, but $f_{.44,.36}(6) = .00$. This shifted the sixth entry in the list from largest to smallest, while keeping all the remaining entries in the same relative order:
\[
(.32, .76, .20, .64, .08, .52, .96, .40) \to (.36, .80, .24, .68, .12, .56, .00, .44),
\]
and so the new permutation is just a cyclic shift of the entries of the old permutation:
\[
 \pi = 42075316 \to 64207531.
\]
To summarize:
\[
 \mbox{vertical translation of the line} \leftrightarrow \mbox{cyclic shift of the permutation}.
\]
A vertical translation of one unit brings the line back to itself mod 1, so eventually these cyclic shifts will return us to the original permutation.

What happens when we vary $\alpha$ is quite a bit more interesting. To get a sense of what happens, it is simplest to consider the $\beta=0$ case. Translating our example with $\alpha=.44$ down to the origin, we have the line $y=.44x$, as drawn with a solid line in Figure \ref{fig:ratslopes}. Through a quick computation (or by our observation about cyclic shifts), we find that this line corresponds to the S\'os permutation $\pi=07531642$.

\begin{figure}
\[
\begin{tikzpicture}
\draw (0,0) node[scale=.75] {
 \begin{tikzpicture}[scale=1.25]
\draw[step=1cm, gray, very thin] (-.9,-.9) grid (7.9,7.9);
\draw[->,very thick] (-1,0) -- (7.9,0) node[right] {$x$};
\draw[->,very thick] (0,-1) -- (0,7.9) node[right] {$y$};
\draw[thick,dashed] (0,0)--(8,8);
\draw[thick,dashed] (0,0)--(8,8/2) node[right,yshift=2mm] {$\frac{1}{2}$};
\draw[thick,dashed] (0,0)--(8,8/3);
\draw[thick,dashed] (0,0)--(8,8/4);
\draw[thick,dashed] (0,0)--(8,8/5);
\draw[thick,dashed] (0,0)--(8,8/6);
\draw[thick,dashed] (0,0)--(8,8/7);
\draw[thick,dashed] (0,0)--(8,16/7);
\draw[thick,dashed] (0,0)--(8,16/5);
\draw[thick,dashed] (0,0)--(8,16/3);
\draw[thick,dashed] (0,0)--(8,24/7) node[right, yshift=-.1cm] {$\frac{3}{7}$};
\draw[thick,dashed] (0,0)--(8,24/5);
\draw[thick,dashed] (0,0)--(8,24/4);
\draw[thick,dashed] (0,0)--(8,32/7);
\draw[thick,dashed] (0,0)--(8,32/5);
\draw[thick,dashed] (0,0)--(8,40/7);
\draw[thick,dashed] (0,0)--(8,40/6);
\draw[thick,dashed] (0,0)--(8,48/7);
\foreach \x in {1,...,7}{
 \draw (\x,0) node[below] {$\x$};
 }
\foreach \x in {1,...,7}{
 \draw (0,\x) node[left] {$\x$};
 }
\draw (1,1) node[circle, fill=black, draw = none, inner sep=2] {};
\draw (2,1) node[circle, fill=black, draw = none, inner sep=2] {};
\draw (3,1) node[circle, fill=black, draw = none, inner sep=2] {};
\draw (4,1) node[circle, fill=black, draw = none, inner sep=2] {};
\draw (5,1) node[circle, fill=black, draw = none, inner sep=2] {};
\draw (6,1) node[circle, fill=black, draw = none, inner sep=2] {};
\draw (7,1) node[circle, fill=black, draw = none, inner sep=2] {};
\draw (7,2) node[circle, fill=black, draw = none, inner sep=2] {};
\draw (5,2) node[circle, fill=black, draw = none, inner sep=2] {};
\draw (3,2) node[circle, fill=black, draw = none, inner sep=2] {};
\draw (7,3) node[circle, fill=black, draw = none, inner sep=2] {};
\draw (5,3) node[circle, fill=black, draw = none, inner sep=2] {};
\draw (4,3) node[circle, fill=black, draw = none, inner sep=2] {};
\draw (7,4) node[circle, fill=black, draw = none, inner sep=2] {};
\draw (5,4) node[circle, fill=black, draw = none, inner sep=2] {};
\draw (7,5) node[circle, fill=black, draw = none, inner sep=2] {};
\draw (6,5) node[circle, fill=black, draw = none, inner sep=2] {};
\draw (7,6) node[circle, fill=black, draw = none, inner sep=2] {};
\draw[red, very thick] (0,0)--(8.5,3.74) node[right,color=black] {$y=.44x$};
\draw (0,0) node[circle, fill=red, draw = none, inner sep=2] {};
\draw (1,.44) node[circle, fill=red, draw = none, inner sep=2] {};
\draw (2,.88) node[circle, fill=red, draw = none, inner sep=2] {};
\draw (3,1.32) node[circle, fill=red, draw = none, inner sep=2] {};
\draw (4,1.76) node[circle, fill=red, draw = none, inner sep=2] {};
\draw (5,2.2) node[circle, fill=red, draw = none, inner sep=2] {};
\draw (6,2.64) node[circle, fill=red, draw = none, inner sep=2] {};
\draw (7,3.08) node[circle, fill=red, draw = none, inner sep=2] {};
\end{tikzpicture}
};
\end{tikzpicture}
\]
\caption{The $\beta=0$ case and rational slopes with denominator at most $n=7$. Each cone corresponds to precisely one permutation. Note that since $\beta=0$, we have $\pi(0)=0$.}\label{fig:ratslopes}
\end{figure}
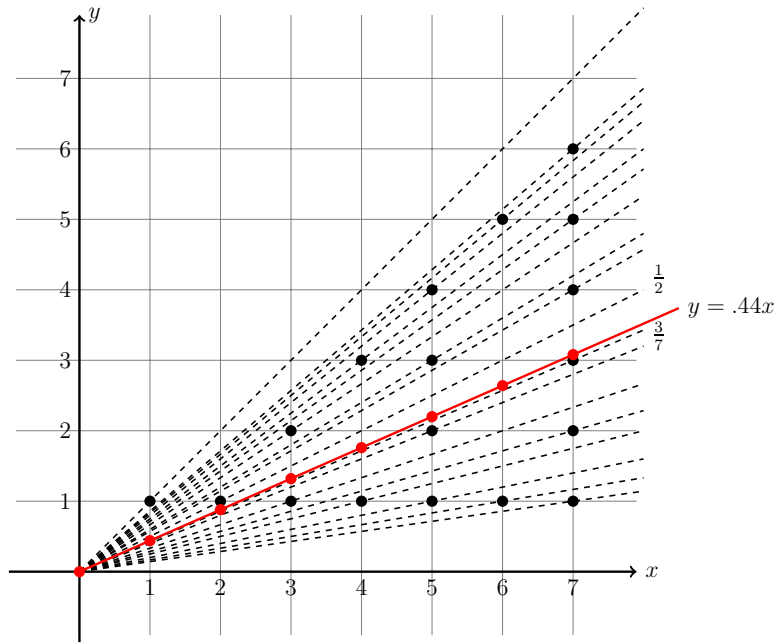

From the picture it seems plausible that a small change in the slope $\alpha$ will not change the permutation $\pi$. In fact, any slope $\alpha$ with $3/7 < \alpha < 1/2$ yields the same permutation $\pi$, though this is not immediately obvious. What is more clear is that increasing the slope to just a bit above $1/2$ would force $f(2)$ to switch from being the largest entry to being the smallest nonzero entry. Similarly we see that decreasing the slope to just a bit below $3/7$ would force $f(7)$ to switch from being the smallest nonzero entry to being the largest.

It is natural to wonder which permutations can arise in this fashion, and we find that it is far from all of them. Rather than all $(n+1)!$ permutations of $\{0,1,2,\ldots,n\}$, there are roughly $n^3$ S\'os permutations. The sequence of numbers $|\Sos_n|$, $n\geq 1$, begins
\[
 2, 6, 16, 30, 60, 84, 144, 198, 280, 352,\ldots.
\]
We see the $n$th term has a factor of $n+1$, which we will show corresponds to the cyclic shifts coming from translation by $\beta$. Now if we compute $|\Sos_n|/(n+1)$, the sequence begins
\[
 1, 2, 4, 6, 10, 12, 18, 22, 28, 32,\ldots,
\]
which is entry A002088 in Sloane's On-Line Encyclopedia of Integer Sequences \cite{OEIS}.  (We have added $|\Sos_n|$ to OEIS as entry A330503.) This entry tells us that these numbers count the number of intervals in the $n$th \emph{Farey sequence}, i.e., the sequence of reduced fractions between $0$ and $1$ with denominators at most $n$.

\subsection*{The three gaps theorem}

This paper is not the first to study S\'os permutations, but few have made them the main focus of study. S\'os permutations seem to have first appeared in work of Vera S\'os from 1958 \cite{Sos}. There the permutations played an important role in her proof of Steinhaus's ``three gaps conjecture,'' which states that (when $\beta=0$) the points $f(0), f(1),\ldots,f(n)$ partition the unit interval into subintervals of at most three different sizes.

\begin{thm}[Three Gaps Theorem]\label{thm:threegaps}
Let $\alpha$ be a real number, let $f(x)= \alpha x \mod 1$, and let $n$ be a natural number. Upon sorting the values $f(0), f(1),\ldots,f(n)$ via
\[
 0 = f(0) \leq f(\pi(1)) \leq \cdots \leq f(\pi(n)) < 1,
\]
the gaps $\delta_k=f(\pi(k+1))-f(\pi(k))$, where $k=0,\ldots,n-1$, and $\delta_n=1-f(\pi(n))$ attain at most three distinct values. In particular, for all $k$,
\[
 \delta_k \in \{ \delta_0, \delta_n, \delta_0+\delta_n \}.
\]
\end{thm}

In the original framing of the problem, the values $\alpha, 2\alpha,\ldots, n\alpha$  were arc lengths measured on a circle of circumference 1. As a shift by $\beta$ would merely correspond to simultaneously rotating these points, the three gaps theorem holds for any $f(x)=\alpha x + \beta \mod 1$, so long as we define the gaps appropriately near 0. See Figure \ref{fig:circle}.

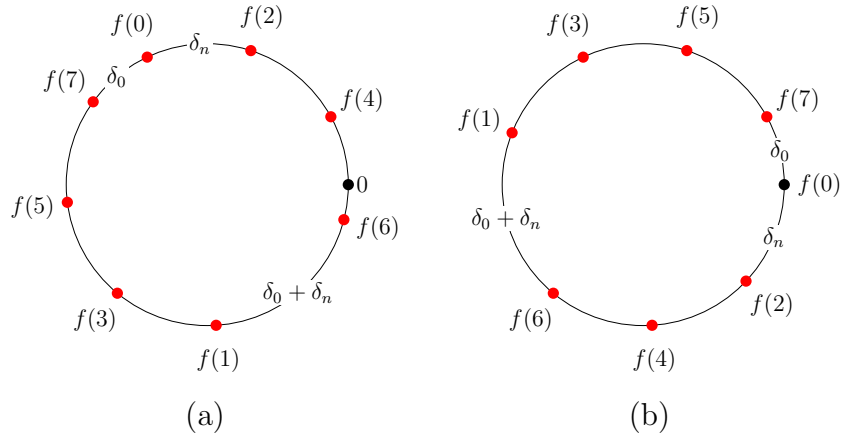
\begin{figure}
\begin{tabular}{c c}
\begin{tikzpicture}
\draw (0,0) node[scale=.75] {
\begin{tikzpicture}[scale=1.25]
\draw (0,0) circle (2);
\draw (0:2) node[fill=black,inner sep=2, circle] {} node[right] {$0$};
\draw (115.2:2) node[fill=red,inner sep=2, circle] {};
\draw (115.2:2.5) node {$f(0)$};
\draw (273.6:2) node[fill=red,inner sep=2, circle] {};
\draw (273.6:2.5) node {$f(1)$};
\draw (72:2) node[fill=red,inner sep=2, circle] {};
\draw (72:2.5) node {$f(2)$};
\draw (230.4:2) node[fill=red,inner sep=2, circle] {};
\draw (230.4:2.5) node {$f(3)$};
\draw (28.8:2) node[fill=red,inner sep=2, circle] {};
\draw (28.8:2.5) node {$f(4)$};
\draw (187.2:2) node[fill=red,inner sep=2, circle] {};
\draw (187.2:2.5) node {$f(5)$};
\draw (345.6:2) node[fill=red,inner sep=2, circle] {};
\draw (345.6:2.5) node {$f(6)$};
\draw (144:2) node[fill=red,inner sep=2, circle] {};
\draw (144:2.5) node {$f(7)$};
\draw (130:2) node[fill=white, inner sep=1] {$\delta_0$};
\draw (93:2) node[fill=white, inner sep=1] {$\delta_n$};
\draw (310:2) node[fill=white, inner sep=1] {$\delta_0+\delta_n$};
\end{tikzpicture}
};
\end{tikzpicture}
&
\begin{tikzpicture}
\draw (0,0) node[scale=.75] {
\begin{tikzpicture}[scale=1.25]
\draw (0,0) circle (2);
\draw (0:2) node[fill=black,inner sep=2, circle] {};
\draw (0:2.5) node {$f(0)$};
\draw (158.4:2) node[fill=red,inner sep=2, circle] {};
\draw (158.4:2.5) node {$f(1)$};
\draw (316.8:2) node[fill=red,inner sep=2, circle] {};
\draw (316.8:2.5) node {$f(2)$};
\draw (115.2:2) node[fill=red,inner sep=2, circle] {};
\draw (115.2:2.5) node {$f(3)$};
\draw (273.6:2) node[fill=red,inner sep=2, circle] {};
\draw (273.6:2.5) node {$f(4)$};
\draw (72:2) node[fill=red,inner sep=2, circle] {};
\draw (72:2.5) node {$f(5)$};
\draw (230.4:2) node[fill=red,inner sep=2, circle] {};
\draw (230.4:2.5) node {$f(6)$};
\draw (28.8:2) node[fill=red,inner sep=2, circle] {};
\draw (28.8:2.5) node {$f(7)$};
\draw (14:2) node[fill=white, inner sep=1] {$\delta_0$};
\draw (-22:2) node[fill=white, inner sep=1] {$\delta_n$};
\draw (194:2) node[fill=white, inner sep=1] {$\delta_0+\delta_n$};
\end{tikzpicture}
};
\end{tikzpicture}
\\
(a) & (b)
\end{tabular}
\caption{The circle model for (a) $f(x)=.44x+.32 \mod 1$ yielding $\pi=42075316$ and (b) $f(x)=.44x \mod 1$ yielding $\pi=07531642$. The gap sizes are clearly the same.}\label{fig:circle}
\end{figure}

The history of the three gaps theorem is quite interesting. Apparently Steinhaus made the conjecture in the mid-1950s, and several people found proofs contemporaneously, with varied techniques. The standard references are papers by Vera S\'os \cite{Sos}, J\'anos Sur\'anyi \cite{Suranyi}, Stanis\l{}aw \'Swierczkowski \cite{swier}, and also John Halton \cite{Halton}, but S\'os and Sur\'anyi also credit Paul Erd\H{o}s, Gy\"orgy Haj\'os, and P\'eter Sz\"usz. Noel Slater had also done work on an equivalent problem in the early 1950s, and outlined the connections between them in a 1967 paper \cite{Slater}. A recent proof using lattice theory comes in a 2017 paper by Jens Marklof and Andreas Str\"ombergsson \cite{MarklofStrombergsson}.

The three gaps theorem is important for its relevance to Diophantine approximation (the gap sizes can be described in terms of convergents of partial fractions), and for its connection to the combinatorics of Sturmian sequences. See \cite{AlessandriBerthe} for a survey of the theorem and some of its applications.

In general, the sequence $\{k \alpha \mod 1: k=1,\ldots,n\}$ has been studied with $\alpha$ fixed and $n$ varying. For example, in \cite{BoydSteele} David Boyd and J. Michael Steele studied the problem of longest increasing subsequences by choosing a fixed irrational $\alpha$ and examining the properties of $\{k \alpha \mod 1: k=1,\ldots,n\}$ as $n\to \infty$. This article takes a different view, with $n$ fixed and $\alpha$ and $\beta$ varying. We are not so much worried about the gaps between the points but the ordering of the points on the unit interval. S\'os gave a precise characterization of these permutations in proving the three gaps theorem, and Sur\'anyi revisited her result in his proof, but also made the connection with Farey sequences. We will review some of their main ideas in this article.

We know of only three papers that study S\'os permutations in their own right. Kevin O'Bryant studies some of the algebraic properties of S\'os permutations, such as their sign, in \cite{Obryant}. In \cite{Cooper}, Joshua Cooper follows up on O'Bryant's work by considering S\'os permutations in relation to quasirandom permutations. Anton Shutov enumerates S\'os permutations in the $\beta=0$ case, in \cite{Shutov}. While this article is partly expository, we do broaden the scope of prior investigations, and we provide self-contained proofs of most results stated. With all the attention that has been given to the three gaps theorem, we think it is time that the set of S\'os permutations gets its day in the sun.

\section{S\'os permutations at $\beta=0$}

In this section we provide some background details and describe the results of Sur\'anyi and S\'os in detail.

\subsection*{Farey sequences}

Here we provide some background for Farey sequences, as found, for example, in Chapter III of Hardy and Wright's book \cite{HardyWright}. The $n$th \emph{Farey sequence}, denoted $F^{(n)}$, is the sequence of all nonnegative fractions $a/b$ such that $0\leq a\leq b \leq n$ and $\gcd(a,b) = 1$, ordered from smallest to largest, e.g.,
\[
F^{(7)} = \left( \frac{0}{1}, \frac{1}{7}, \frac{1}{6}, \frac{1}{5}, \frac{1}{4}, \frac{2}{7}, \frac{1}{3}, \frac{2}{5}, \frac{3}{7}, \frac{1}{2}, \frac{4}{7}, \frac{3}{5}, \frac{2}{3}, \frac{5}{7}, \frac{3}{4}, \frac{4}{5}, \frac{5}{6}, \frac{6}{7}, \frac{1}{1} \right).
\]
There is a straightforward correspondence between the fractions $a/b$ in $F^{(n)}$ and the lines of rational slope, $by=ax$, in the $n \times n$ grid that pass through integer lattice points $(kb,ka)$. These are illustrated with the dots on the dashed lines of Figure \ref{fig:ratslopes}.

An interval $(a/b,c/d)$ is called a \emph{Farey interval} if $a/b$ and $c/d$ are consecutive terms in $F^{(n)}$ for some $n$. In fact, such consecutive terms enjoy many wonderful properties, some of which we collect here.

\begin{prop}\label{prp:fareyfacts}
Suppose $a/b < c/d$ are consecutive terms in $F^{(n)}$. Then:
\begin{enumerate}
\item If $n>1$, then $b\neq d$.\label{prp:fareyfacts1}
\item $bc-ad=1$, i.e., $c/d-a/b=1/bd$.\label{prp:fareyfacts2}
\item If $a'/b < c'/d$ are consecutive terms, then $a'=a$ and $c'=c$, i.e., denominators determine intervals.\label{prp:fareyfacts3}
\item $b+d>n$.\label{prp:fareyfacts4}
\item $a/b < (a+c)/(b+d) < c/d$ are three consecutive terms in $F^{(b+d)}$.\label{prp:fareyfacts5}
\end{enumerate}
\end{prop}

The middle fraction in part \ref{prp:fareyfacts5} of Proposition \ref{prp:fareyfacts} is known as the \emph{mediant} $m=(a+b)/(c+d)$ of the interval $(a/b,c/d)$. For example, $m=5/12$ is the mediant of the interval $(2/5,3/7)$.

All the fractions in $F^{(n-1)}$ are again fractions in $F^{(n)}$, and the only way to obtain new fractions is to consider those fractions of the form $a/n$, where $a< n$ and $\gcd(a,n) =1$. Thus,
\[
 |F^{(n)}| = |F^{(n-1)}| + \varphi(n),
\]
where $\varphi(n)$ is Euler's totient function. From this recurrence, it follows by induction that the cardinality of the $n$th Farey sequence is given by
\[
 |F^{(n)}| = 1 + \sum_{k=1}^n \varphi(k),
\]
and the number of Farey intervals is
\begin{equation}\label{eq:fareycount}
 |F^{(n)}| -1=\sum_{k=1}^n \varphi(k).
\end{equation}
A bit of analytic number theory can be used to show this quantity is asymptotically $3n^2/\pi^2$ (see, e.g., \cite[Section 18.5]{HardyWright}).

\subsection*{Sur\'anyi's bijection}

Throughout this section, we focus on the case of $\beta=0$ and explain Sur\'anyi's bijection between the S\'os permutations $\pi_{\alpha,0}^{(n)}$ and intervals in the $n$th Farey sequence. We suppress $\beta=0$ from the notation and simply write $\pi=\pi_{\alpha}^{(n)}$ for now.

Recall that when $n$ is given, any permutation $\pi$ is determined uniquely by its \emph{inversion set}, given by
\[
 I(\pi)=\{ (i,j) : i<j \mbox{ and } \pi(i)>\pi(j)\}.
\]
We will describe S\'os permutations in terms of their inversion sets.

Fix a slope $\alpha$ and an integer $n$.  Let $f=f_{\alpha}(x) = \alpha x \mod 1$ and let $\pi=\pi_{\alpha}^{(n)}$ denote the corresponding S\'os permutation. Here we have $\pi(i) > \pi(j)$ if and only if $f(i)>f(j)$. Thus,
\[
 (i,j) \in I(\pi) \Leftrightarrow f(j)-f(i) < 0.
\]

For any fixed pair $(i,j)$ with $i<j$, define the function $h(\alpha)=f_{\alpha}(j)-f_{\alpha}(i)$. We make a few easy, but important, observations about the function $h$.

\begin{obs}\label{obs:h}
Fix a pair of integers $(i,j)$ with $1\leq i<j \leq n$ and consider the function $h(\alpha)=f_{\alpha}(j)-f_{\alpha}(i)$ on the interval $(0,1)$. Then:
\begin{enumerate}
\item $h$ is a piecewise linear function of slope $(j-i)$.\label{obs:h1}
\item If $h$ is discontinuous at $\alpha$, then $\alpha = a/b$ is rational and either $i$ is a multiple of $b$ or $j$ is a multiple of $b$, but not both. \label{obs:h2}
\item If $h(\alpha)=0$, then $\alpha = a/(j-i)$ for some integer $a$. \label{obs:h3}
\end{enumerate}
\end{obs}

Suppose $\alpha\neq \alpha'$ and $\pi_{\alpha}^{(n)} \neq \pi_{\alpha'}^{(n)}$. Then there is some pair $(i,j)$ that is an inversion for, say, $\pi_{\alpha}^{(n)}$ but not for $\pi_{\alpha'}^{(n)}$. This means we have
\[
 h(\alpha) = f_{\alpha}(j) - f_{\alpha}(i) < 0 < f_{\alpha'}(j) - f_{\alpha'}(i) = h(\alpha').
\]
Thus by Observation \ref{obs:h} part \ref{obs:h1}, we know that between $\alpha$ and $\alpha'$ there exists a number $r$ such that either $h(r)=0$ or $h$ is discontinuous at $r$. But by Observation \ref{obs:h} parts \ref{obs:h2} and \ref{obs:h3}, this number $r$ must be rational with denominator at most $j\leq n$, i.e., $r=a/b$ is an element of the Farey sequence $F^{(n)}$. In other words, if $\pi_{\alpha}^{(n)}\neq \pi_{\alpha'}^{(n)}$, then $\alpha$ and $\alpha'$ belong to different Farey intervals. We record the contrapositive of this statement in the following proposition.

\begin{prop}\label{prp:interval}
If $\alpha$ and $\alpha'$ lie in the same Farey interval, then $\pi^{(n)}_{\alpha} = \pi^{(n)}_{\alpha'}$.
\end{prop}

Now fix a Farey interval $(a/b,c/d)$. Proposition \ref{prp:interval} says  $\pi=\pi_{\alpha}^{(n)}$ is constant for all $a/b< \alpha < c/d$. We make the following observations about $\pi$.

For one thing, imagine $\alpha=a/b+\epsilon$ for some $\epsilon < 1/bn$. Then $f(b) = b\epsilon < 1/n$ is the smallest nonzero entry in $(f(0),f(1),\ldots,f(n))$ since $f(kb) =kb\epsilon > f(b)$ for all positive integers $k<n/b$ and for any $i$ that is not a multiple of $b$, $f(i) = ai/b + i\epsilon >$\ some nonzero integer multiple of $1/b$ \ $\geq 1/n > b\epsilon =f(b)$.  As $f(b)$ is the smallest nonzero entry, we have $\pi(1) =b$.

Now imagine $\alpha= c/d-\epsilon$ for some $\epsilon < 1/dn$. Then by reasoning similar to the above, $f(d)=1-d\epsilon$ is the largest entry in $(f(0),f(1),\ldots,f(n))$, and we can conclude that $\pi(n)=d$.

We summarize this discussion in the next proposition.

\begin{prop}\label{prp:denoms}
If $\alpha$ lies in the Farey interval $(a/b,c/d)$, then $\pi(1) = b$ and $\pi(n)=d$.
\end{prop}

Now, according to part \ref{prp:fareyfacts3} of Proposition \ref{prp:fareyfacts}, $b$ and $d$ can be adjacent denominators in at most one Farey interval with $a/b<c/d$. Thus, we know that two different Farey intervals cannot correspond to the same permutation. Taking Proposition \ref{prp:interval} and Proposition \ref{prp:denoms} together, we have proved Sur\'anyi's bijection between Farey intervals and S\'os permutations with $\beta=0$.

\begin{thm}[Sur\'anyi \cite{Suranyi}]\label{thm:suranyi}
There is a bijection between Farey intervals and S\'os permutations with $\beta=0$.
\end{thm}

Also from Sur\'anyi's bijection, we obtain as a corollary that $\pi_{\alpha,0}^{(n)}$ uniquely determines $\pi_{\alpha,0}^{(n+1)}$ if $n+1<b+d$. Otherwise, if $n+1=b+d$, then $\pi_{\alpha,0}^{(n)}$ determines two permutations, depending on whether $\alpha$ is to the left or right of the mediant. We have $\pi(1) = b$ and $\pi(n+1) = b+d$ if $\alpha \in (a/b,(a+b)/(c+d))$, and $\pi(1) = b+d$ and $\pi(n+1)=d$ if $\alpha \in ((a+b)/(c+d), c/d)$.

\subsection*{S\'os's recurrence and proof of the three gaps theorem}

We now present S\'os's theorem about the structure of S\'os permutations, which makes explicit the bijection that was implicit in Theorem \ref{thm:suranyi}.

\begin{thm}[{S\'os \cite[Theorem 1]{Sos}}]\label{thm:sos}
 Suppose $\pi \in \Sos_n$ is a S\'os permutation with $\pi(0)=0$, corresponding to the Farey interval $(a/b,c/d)$. Then, for $0\leq k\leq n-1$,
 \begin{equation}\label{eq:Sos}
  \pi(k+1) = \pi(k)+\begin{cases} b & \mbox{if } \pi(k)\leq n-b,\\
     b-d & \mbox{if } n-b < \pi(k) < d, \\
     -d & \mbox{if } d\leq \pi(k).
     \end{cases}
 \end{equation}
\end{thm}

\begin{proof}
By Theorem \ref{thm:suranyi}, we can choose any $\alpha$ in the interval $(a/b,c/d)$ to define $\pi$. A nice choice is to let $\alpha = (a+c)/(b+d)$ be the mediant of the interval and examine $\pi = \pi_{\alpha,0}^{(n)}$ in this case.

Here we have $f(x) = \alpha x \mod 1$ expressed very simply in terms of modular arithmetic:
\[
 f(i) = \frac{ i\cdot(a+c) \pmod{b+d}}{b+d}.
\]
For example, with $(a/b,c/d)=(2/5,3/7)$ and $n=9$, we take $\alpha = 5/12$ and find that
\[
(f(0),f(1),\ldots,f(9))=\frac{1}{12}(0,5,10,3,8,1,6,11,4,9).
\]

It follows from part \ref{prp:fareyfacts2} of Proposition \ref{prp:fareyfacts} that $b(a+c) \equiv 1 \pmod{b+d}$ and $d(a+c)\equiv -1 \pmod{b+d}$.  Thus $f(b) = 1/(b+d)$ and $f(d) = 1-1/(b+d)$, or $f(-d) = 1/(b+d)$.  Since we can get no closer than 1 modulo $b+d$, we know right away that $\pi(1) = b$ and $\pi(n)=d$. Sorting the example above gives
{\small
\[
 (f(0),f(5),f(3),f(8),f(1),f(6),f(4),f(9),f(2),f(7))=\frac{1}{12}(0,1,3,4,5,6,8,9,10,11).
\]
}

Since $(a+c)$ is relatively prime to $(b+d)$, all the entries in
\[
\mathbf{f}=(f(0),f(\pi(1)), \ldots, f(\pi(n)))
\]
must be distinct. Moreover, since $b+d > n \geq \max\{b,d\} \geq (b+d)/2$, the pigeonhole principle implies that the gaps between consecutive entries, $f(\pi(k+1))-f(\pi(k))$, must be either $1/(b+d)$ or $2/(b+d)$.

First, suppose
\[
 f(\pi(k+1))-f(\pi(k)) = \frac{1}{b+d}.
\]
If $\pi(k+1)>\pi(k)$, then
\[
f\left(\pi(k+1)-\pi(k)\right)=f(\pi(k+1))-f(\pi(k))=f(b).
\]
Since all the entries in $\mathbf{f}$ are distinct, we have $\pi(k+1)-\pi(k) = b$, or $\pi(k+1)=\pi(k)+b$. Of course, this identity can only be true if $\pi(k) \leq n-b$, explaining the first case in equation (\ref{eq:Sos}).

The third case is similar. Suppose $f(\pi(k+1))-f(\pi(k)) = 1/(b+d)$ and now $\pi(k)>\pi(k+1)$. Then $f\left(\pi(k+1)-\pi(k)\right)=f(\pi(k+1))-f(\pi(k))=f(-d)$, and so $\pi(k+1) = \pi(k)-d$, which only makes sense if $\pi(k)\geq d$.

As the middle case of $n-b<\pi(k)<d$ (which is empty if $b+d=n+1$) cannot correspond to a gap of size $1/(b+d)$, it must come from a gap of size $2/(b+d)$. Note that
\[
 \frac{2}{b+d} = f(b)+f(-d) = f(b-d).
\]
Thus, if $f(\pi(k+1)-\pi(k)) = f(\pi(k+1))-f(\pi(k)) =  2/(b+d)$, we have $\pi(k+1)-\pi(k) = b-d$, and therefore $\pi(k+1)=\pi(k)+b-d$.
\end{proof}

From S\'os's recurrence for the entries of $\pi$, we can see precisely how to construct the permutation corresponding to a Farey interval $(a/b,c/d)$. For example, let $\pi$ be the S\'os permutation corresponding to the interval $(3/7,1/2)$ and $n=7$. We have $b=7, d=2$, and $n-b=0$. Thus:
\begin{align*}
 \pi(0) &=0,\\
 \pi(1) &=0+7 =7,\\
 \pi(2) &=7-2 = 5,\\
 \pi(3) &=5-2 = 3,\\
 \pi(4) &=3-2 =1,\\
 \pi(5) &=1+(7-2) = 6,\\
 \pi(6) &=6-2 = 4,\\
 \pi(7) &=4-2 = 2,
\end{align*}
so $\pi = 07531642$.

Moreover, the three gaps theorem (Theorem \ref{thm:threegaps}) now follows. If $\delta_0 = f(\pi(1))-f(\pi(0))=f(b)$ and $\delta_n = 1-f(\pi(n))=1-f(d)=f(-d)$, then Theorem \ref{thm:sos} implies
\[
 \delta_k = f(\pi(k+1))-f(\pi(k)) = \begin{cases}
  \delta_0 & \mbox{if } \pi(k) \leq n-b,\\
  \delta_0+\delta_n & \mbox{if } n-b < \pi(k) < d,\\
  \delta_n & \mbox{if } d\leq \pi(k)
 \end{cases}
\]
for all intergers $0\leq k < n$.

The second case above, with a gap of size $\delta_0+\delta_n$, occurs for some $k$ if and only if $n+1 < b+d$. Thus we can make the following observation.

\begin{obs}\label{obs:distinct}
Suppose $\alpha$ lies in the Farey interval $(a/b,c/d)$ with $f(x) = \alpha x \mod 1$. If $n+1<b+d$, then the set of gaps $\{\delta_k = f(\pi(k+1))-f(\pi(k)) : k=0,1,\ldots,n\}$ has three elements, while if $b+d=n+1$, the set has only two elements: $\delta_0$ and $\delta_n$. Moreover, if $\alpha = a/b$ with $b\leq n$, then $\delta_k \in \{ 0, 1/b\}$.
\end{obs}

\section{New results for S\'os permutations}

In this section we consider S\'os permutations for any $\beta$, and provide some new results, including the enumeration of S\'os permutations, a characterization of the domain of a S\'os permutation, and a ``three areas'' theorem that echoes the three gaps theorem.

\subsection*{Enumeration of S\'os permutations}

Our first main result is an enumeration of S\'os permutations. This follows by combining Sur\'anyi's bijection in the $\beta=0$ case with the cyclic permutation coming from vertical shifts by larger $\beta$.

To be more precise, define, for any fixed $n$ and $\alpha$,
\[
 \Sos_{n,\alpha} = \{ \pi_{\alpha,\beta}^{(n)} : 0\leq \beta < 1 \}.
\]
Let $c$ be the cycle defined by $c(i)=i-1 \pmod{n+1}$, so that $\pi\circ c(i) = \pi(i-1)$ for $i > 0$ and $\pi\circ c(0) =\pi(n)$. For example, $203154\circ c = 420315$.

\begin{lem}\label{lem:cycle}
Let $\pi \in \Sos_{n,\alpha}$. Then
\[
 \Sos_{n,\alpha} = \{ \pi\circ c^k : k = 0,1,\ldots,n\}.
\]
In particular, $|\Sos_{n,\alpha}| = n+1$.
\end{lem}

We leave the details of the proof of Lemma \ref{lem:cycle} to the reader, using the illustration of points on a circle in Figure \ref{fig:circle} as a guide.

By Lemma \ref{lem:cycle}, there is precisely one permutation in $\Sos_{n,\alpha}$ with $\pi(0)=0$ and thanks to Sur\'anyi's bijection, the sets $\Sos_{n,\alpha}$ and $\Sos_{n,\alpha'}$ coincide if $\alpha, \alpha'$ lie in the same Farey interval, and are disjoint otherwise. The number of such intervals is given by equation \eqref{eq:fareycount}, and multiplying by $(n+1)$ we get the following counting formula.

\begin{thm}\label{thm:enumerate}
The number of S\'os permutations of $\{0,1,\ldots,n\}$ is
\[
 |\Sos_n| = (n+1)\sum_{k=1}^n \varphi(k),
\]
which is asymptotically $3n^3/\pi^2$.
\end{thm}

\subsection*{Domains of S\'os permutations}

Our next result provides a two-dimensional analogue of Sur\'anyi's bijection. That is, for any $n$, we provide a partition of the $(\alpha,\beta)$-parameter space into polygonal regions that are in bijection with the elements of $\Sos_n$.

To make this idea a bit more clear, define, for any permutation $\pi$ of $\{0,1,\ldots,n\}$, the \emph{domain} of $\pi$, denoted $S(\pi)$, to be the set of all points in $[0,1)\times [0,1)$ that give rise to $\pi$, i.e.,
\[
S(\pi) = \{ (\alpha,\beta) \in [0,1)^2 : \pi_{\alpha,\beta}^{(n)} = \pi\}.
\]
(The set $S(\pi)$ is nonempty if and only if $\pi$ is a S\'os permutation.) See Figures \ref{fig:F3regions} and \ref{fig:F4regions} for the $n=2$ and $n=3$ examples.

\begin{figure}
\[
\begin{tikzpicture}[scale=6]
\draw[fill=white!50!black, draw=none, opacity=.5] (0,0)--(1,0)--(1,1)--(0,1)--(0,0);
\draw[very thick, ->] (0,-.1)--(0,1.1) node[left] {$\beta$};
\draw[very thick, ->] (-.1,0)--(1.1,0) node[below] {$\alpha$};
\draw (0.5,1)--(0.5,0) node[below] {$\frac{1}{2}$};
\draw (0,1) node[left] {$1$} --(1,0) node[below] {$1$};;
\draw (0,1)--(.5,0);
\draw (.5,1)--(1,0);
\draw (.2,.2) node[fill=none,draw=none] {$012$};
\draw (.4,.45) node[fill=none,draw=none] {$201$};
\draw (.35,.8) node[fill=none,draw=none] {$120$};
\draw (.8,.8) node[fill=none,draw=none] {$210$};
\draw (.6,.55) node[fill=none,draw=none] {$102$};
\draw (.65,.2) node[fill=none,draw=none] {$021$};
\end{tikzpicture}
\]
\caption{The 6 regions corresponding to the S\'os permutations with $n=2$.}\label{fig:F3regions}
\end{figure}
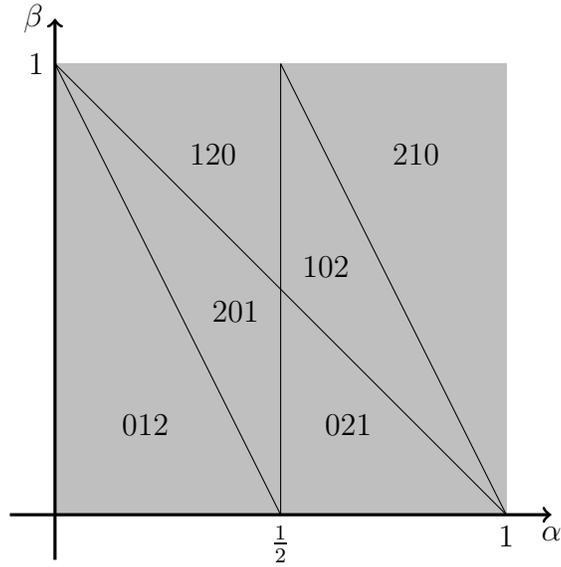

\begin{figure}
\[
\begin{tikzpicture}[yscale=8, xscale=8]
\draw[fill=white!50!black, draw=none, opacity=.5] (0,0)--(1,0)--(1,1)--(0,1)--(0,0);
\draw[very thick, ->] (0,-.1)--(0,1.1) node[left] {$\beta$};
\draw[very thick, ->] (-.1,0)--(1.1,0) node[below] {$\alpha$};
\draw (0.3333,1)--(0.3333,0) node[below] {$\frac{1}{3}$};
\draw (0.5,1)--(0.5,0) node[below] {$\frac{1}{2}$};
\draw (0.6666,1)--(0.6666,0) node[below] {$\frac{2}{3}$};
\draw[dashed] (0.25,1)--(0.25,0) node[below] {$\frac{1}{4}$};
\draw[dashed] (0.75,1)--(0.75,0) node[below] {$\frac{3}{4}$};
\draw (0,1) node[left] {$1$}--(1,0) node[below] {$1$};
\draw (0,1)--(.5,0);
\draw (0,1)--(.3333,0);
\draw[dashed] (0,1)--(0.25,0);
\draw[dashed] (0.25,1)--(0.5,0);
\draw[dashed] (0.5,1)--(0.75,0);
\draw[dashed] (0.75,1)--(1,0);
\draw (.5,1)--(1,0);
\draw (.6666,1)--(1,0);
\draw (.3333,1)--(.6666,0);
\draw (.15,.15) node[fill=none,draw=none] {$0123$};
\draw (.27,.3) node[fill=none,draw=none,rotate=-70] {$3012$};
\draw (.25,.62) node[fill=none,draw=none,rotate=-60] {$2301$};
\draw (.2,.9) node[fill=none,draw=none] {$1230$};
\draw (.4,.05) node[fill=none,draw=none] {$0312$};
\draw (.42,.42) node[fill=none,draw=none] {$2031$};
\draw (.4,.65) node[fill=none,draw=none, rotate=-45] {$1203$};
\draw (.43,.95) node[fill=none,draw=none] {$3120$};
\draw (.57,.05) node[fill=none,draw=none] {$0213$};
\draw (.6,.35) node[fill=none,draw=none, rotate=-45] {$3021$};
\draw (.58,.58) node[fill=none,draw=none] {$1302$};
\draw (.6,.95) node[fill=none,draw=none] {$2130$};
\draw (.8,.1) node[fill=none,draw=none] {$0321$};
\draw (.75,.38) node[fill=none,draw=none,rotate=-60] {$1032$};
\draw (.73,.7) node[fill=none,draw=none,rotate=-70] {$2103$};
\draw (.85,.85) node[fill=none,draw=none] {$3210$};
\end{tikzpicture}
\]
\caption{The 16 regions corresponding to the S\'os permutations with $n=3$, along with the dashed lines that will give the partition for $n=4$.}\label{fig:F4regions}
\end{figure}
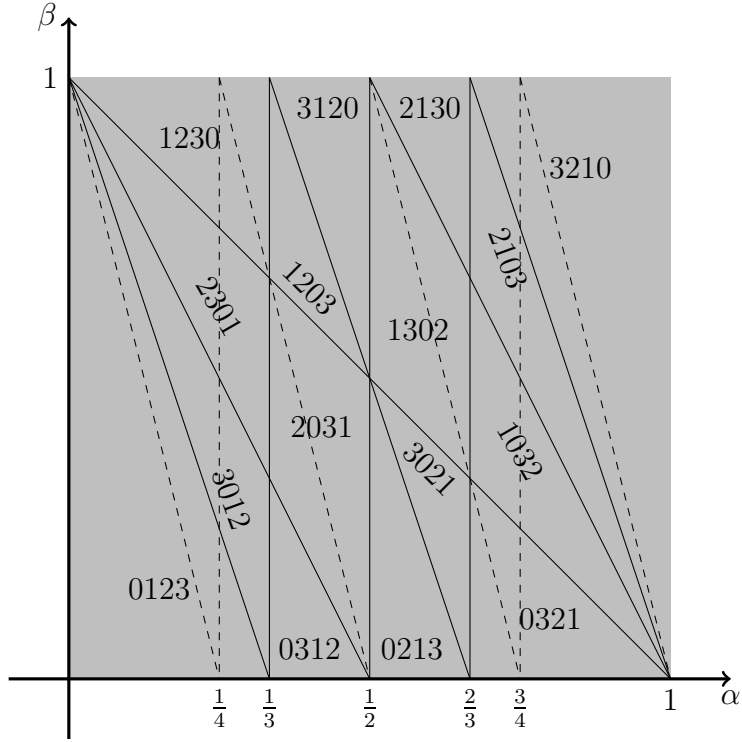

We have the following simple characterization of the regions $S(\pi)$.\footnote{Curiously, these lines are a subset of the ``tune resonance lines'' from accelerator physics, as described in \cite{tomas}.
}

\begin{thm}\label{thm:domains}
For each $\pi \in \Sos_n$, the interior of the domain $S(\pi)$ is a connected component in the complement of the lines
\[
 \{ \alpha=a/b : a/b \in F^{(n)} \} \cup \{ i\alpha+\beta = j : 1\leq j\leq i \leq n \}.
\]
\end{thm}

Moreover, given a S\'os permutation $\pi$ we can determine the boundary of its domain $S(\pi)$ as follows. Suppose $k$ is such that $\pi(k)=0$. Then we let $b$ be the entry to the right of zero and $d$ be the entry to the left of zero.  That is, we let
\[
d = \pi(k-1), \qquad b=\pi(k+1).
\]
These entries are taken cyclically, so that if $k=0$ then $\pi(k-1) = \pi(n)$, and if $k=n$ then $\pi(k+1)=\pi(0)$. By Lemma \ref{lem:cycle} and Sur\'anyi's bijection, we see that these two entries of $\pi$ determine the Farey interval $(a/b,c/d)$ in which $\alpha$ must lie.

Now let
\begin{equation}\label{eq:jtop}
j_{\bt} = \begin{cases}
      0 & \mbox{if }\pi(0)=0 \\
      1+\left\lfloor \frac{a}{b}\pi(0) \right\rfloor & \mbox{if } \pi(0)\neq 0
   \end{cases}
    \quad \mbox{ and } \quad j_{\tp} = 1+\left\lfloor \frac{a}{b}\pi(n) \right\rfloor.
\end{equation}
Then for any $\alpha \in (a/b,c/d)$, we show that it must be the case that
\[
 j_{\bt}-\pi(0)\alpha \leq  \beta < j_{\tp}-\pi(n)\alpha.
\]
Notice that if $\beta = j_{\bt}-\pi(0)\alpha$, then $\pi(0)\alpha+\beta=j_{\bt}$ is an integer, and so $f_{\alpha,\beta}(\pi(0)) = 0$. Likewise, if $\beta=j_{\tp}-\pi(n)\alpha$, then we have $f_{\alpha,\beta}(\pi(n)) = 0$.

In other words, the domain $S(\pi)$ is the region bounded by lines $\alpha=a/b$, $\alpha=c/d$, $\pi(0)\alpha+\beta = j_{\bt}$, and $\pi(n)\alpha+\beta = j_{\tp}$. Two such oblique lines intersect at $\alpha = (j_{\tp}-j_{\bt})/(\pi(n)-\pi(0))$. Since all the integers in this expression are at most $n$, any points of intersection for the lines must occur at rational numbers whose reduced form has denominator at most $n$. That is, the intersection of the lines cannot occur in the interior of a Farey interval.
The region $S(\pi)$ is a triangle if $(j_{\tp}-j_{\bt})/(\pi(n)-\pi(0))$ equals $a/b$ or $c/d$; otherwise, it is a trapezoid.

\begin{thm}\label{thm:pidomain}
Let $\alpha \in (a/b,c/d)$ be a Farey interval and let $\pi = \pi_{\alpha,\beta}^{(n)} \in \Sos_n$. The domain $S(\pi)$ is bounded on the left by the line $\alpha=a/b$, on the right by the line $\alpha=c/d$, below by the line $\pi(0)\alpha+\beta=j_{\bt}$, and above by the line $\pi(n)\alpha+\beta=j_{\tp}$. The region $S(\pi)$ has area
 \[
  \frac{1}{bd}\left(j_{\tp}-j_{\bt}+\left(\frac{a}{b}+\frac{c}{d}\right)\left(\frac{\pi(0)-\pi(n)}{2}\right)\right).
 \]
\end{thm}

For example, suppose we pick a random S\'os permutation with $n=9$ by choosing the point $(\alpha,\beta)=(.42,.31)$. We find $\pi=9 2 7 0 5 3 8 1 6 4$, and we wish to describe the region containing all points that correspond to this permutation. We have $b=5$ and $d=7$, and from the Euclidean algorithm we find $a=2$ and $c=3$, so we know $\alpha$ lives in the Farey interval $(2/5,3/7)$. This allows us to compute $j_{\bt} = 1 + \lfloor (2/5)\cdot 9\rfloor = 4$ and $j_{\tp}=1+\lfloor (2/5)\cdot 4\rfloor =2$. This gives the region bounded on the left by $\alpha =2/5$, on the right by $\alpha=3/7$, below by $9\alpha+\beta = 4$, and above by $4\alpha + \beta = 2$. The area of this region is
 \[
  \frac{1}{35}\left(2-4+\left(\frac{2}{5}+\frac{3}{7}\right)\left(\frac{9-4}{2}\right)\right) = \frac{1}{490}=.002040816\ldots,
 \]
so the probability of a random point in $[0,1)^2$ yielding this permutation is a bit more than $0.2\%$.

Theorem \ref{thm:domains} follows from Theorem \ref{thm:pidomain} by letting the latter theorem range over all $\pi\in \Sos_n$, so we focus on Theorem \ref{thm:pidomain}.

Let $\pi =\pi_{\alpha,\beta}^{(n)}$ be a S\'os permutation. By Lemma \ref{lem:cycle}, we know $\pi = \pi'\circ c^k$ for some integer $k$, where $\pi'=\pi_{\alpha,0}^{(n)}$ and $c$ is the cyclic shift operator. In particular, $\pi'(0)=0$, and so $\pi(k) = 0$. By Theorem \ref{thm:suranyi}, we know $\alpha$ lies in a unique Farey interval $(a/b,c/d)$. Moreover, by Proposition \ref{prp:denoms}, we know $b=\pi'(1)=\pi(k+1)$ and $d=\pi'(n)=\pi'(k-1)$.

Having established the interval in which $\alpha$ lives in terms of $\pi$, we move on to the relationship between $k$ and $\beta$. To finish off Theorem \ref{thm:pidomain}, we need to prove the following proposition.

\begin{prop}\label{prp:toplines}
Suppose $\pi=\pi_{\alpha,\beta}^{(n)}$ is a S\'os permutation corresponding to the Farey interval $(a/b,c/d)$. Then
\[
j_{\bt} - \pi(0)\alpha \leq \beta < j_{\tp} - \pi(n)\alpha,
\]
where $j_{\bt}$ and $j_{\tp}$ are as defined in \eqref{eq:jtop}.
\end{prop}

Our proof will rely on the following lemma.

\begin{lem}\label{lem:floor}
For any $i=0,1,\ldots,n$, and any $\alpha$ in the Farey interval $(a/b,c/d)$, $1-(i\alpha \mod 1) = 1+\lfloor i\cdot a/b\rfloor-i\alpha$.
\end{lem}

\begin{proof}
It is clearly equivalent to prove the following:
\[
 i\alpha-\left\lfloor i\cdot \frac{a}{b}\right\rfloor = i \alpha \mod 1,
\]
which follows if $\lfloor i \cdot a/b  \rfloor = \lfloor i\alpha \rfloor$. Let $j=\lfloor i\cdot a/b\rfloor$. Since $a/b < \alpha < c/d$, we have $j\leq i\cdot a/b < i\cdot \alpha < i\cdot c/d$. We will prove that $i\cdot c/d \leq j+1$, from which the lemma follows.

Suppose not. If $i\cdot c/d > j+1$, then there exists an $r$ such that $a/b < r < c/d$ and $i\cdot r = j+1$. But then $r=(j+1)/i$ is a rational number between $a/b$ and $c/d$ with denominator $i \leq n$, contradicting the assumption that $(a/b,c/d)$ is an interval in $F^{(n)}$.
\end{proof}

\begin{proof}[Proof of Proposition \ref{prp:toplines}]
Let $\pi=\pi_{\alpha,\beta}^{(n)}$ be a S\'os permutation with $\pi(k)=0$, i.e., let $\pi = \pi'\circ c^k$, where $\pi'(0)=0$.

Let $f$ be the function $f=f_{\alpha,0}= \alpha x \mod 1$, and let $d_i = 1-f(\pi'(n-i))$, where $d_{-1}=0$. Then it is straightforward to verify that, for $\beta' \in [d_{i-1}, d_i)$, we have $\pi_{\alpha,\beta'}^{(n)} = \pi'\circ c^i$. Thus $d_{k-1} \leq \beta < d_k$.

But $d_{k-1} = 1-f(\pi'(n-k+1)) = 1-f(\pi(0))$ and $d_k = 1-f(\pi'(n-k))=1-f(\pi(n))$, and the proposition now follows from Lemma \ref{lem:floor}.
\end{proof}

As a corollary to Theorem \ref{thm:pidomain}, we can see that a S\'os permutation $\pi \in \Sos_n$ determines either one, two, or three permutations in $\Sos_{n+1}$, according to whether the domain $S(\pi)$ contains portions of the lines $\alpha = a/(n+1)$ with $\gcd(a,n+1)=1$, or $(n+1)\alpha + \beta = j$ for some $j=1,\ldots,n+1$. For example, in passing from $n=2$ to $n+1=3$ we see $S(120) = S(1230) \cup S(1203) \cup S(3120)$, while $S(201)=S(2301)\cup S(2031)$. However, starting with $n=3$, we find no lines involving $n+1=4$ that pass through $S(3120)$, so $S(3120)=S(31420)$. See Figures \ref{fig:F3regions} and \ref{fig:F4regions}.

\subsection*{The three areas theorem}

Theorem \ref{thm:pidomain} also leads us to a result that is reminiscent of the three gaps theorem. It turns out that among all S\'os permutations $\pi$ whose domain lies in the vertical strip $(a/b, c/d)\times [0,1)$, there are at most three possible values for the area of the domains $S(\pi)$.

In order to prove this, we first revisit the idea in the proof of Proposition \ref{prp:toplines} with a fixed $\alpha$. We let $f(x) = \alpha x \mod 1$ and let $\pi = \pi_{\alpha,0}$. Further we define $\beta_k = 1-f(\pi(k))$, for each $k=0,1,\ldots,n-1$, to be the vertical coordinates of the points at which we cross the diagonal domain boundary lines. Then we have
\[
 \beta_k - \beta_{k+1} = f(\pi(k+1)) - f(\pi(k)) = \delta_k.
\]
That is, the gaps between the diagonal lines intersecting the vertical line at $\alpha$ are precisely the gaps $\delta_k$ from the three gaps theorem.

\begin{obs}\label{lem:vert_coord}
Let $\alpha' \in [a/b,c/d)$. Then the diagonal lines of the form $i\alpha +\beta =j$ with $1\leq j\leq i \leq n$ that cross the vertical line $\alpha=\alpha'$ with $\beta \in [0,1)$ do so with vertical coordinates $\beta_k = \sum_{j=k}^n \delta_j$. In particular, if $\alpha'=a/b$, these coordinates are $1/b, 2/b,\mathellipsis, 1$.
\end{obs}

We already know from Theorem \ref{thm:pidomain} that a domain $S(\pi)$ can be either a trapezoid or a triangle, and Observation \ref{lem:vert_coord} implies that the crossings of the vertical lines $\alpha = a/b$ and $\alpha = c/d$ are evenly spaced, of sizes $1/b$ and $1/d$, respectively.  Combining these ideas, we obtain the following theorem.

\begin{thm}[Three Areas Theorem] \label{thm:3probs}
Let $(a/b,c/d)$ be an interval in the Farey sequence $F^{(n)}$. Among all S\'os permutations $\pi \in \Sos_n$ whose domain $S(\pi)$ lies in the vertical strip $(a/b, c/d)\times [0,1)$, there are three possible areas for $S(\pi)$, and one of them is the sum of the other two. In particular, let $w=c/d-a/b=1/bd$ be the width of the Farey interval. Then the area of $S(\pi)$ is
\begin{itemize}
\item $d\cdot w^2/2$ when the region is a triangle with one side along the vertical line $\alpha=a/b$,
\item $b \cdot w^2/2$ when the region is a triangle with one side along the vertical line $\alpha=c/d$, and
\item $(b+d)\cdot w^2/2$ when the region is a trapezoid.
\end{itemize}
\end{thm}

With the three areas theorem in hand, we can deduce a number of further consequences for S\'os permutations.

For example, taking the uniform distribution on $[0,1)^2$, we can say which S\'os permutations occur with the greatest and least probability by looking at largest and smallest domains. Since $1\leq b,d\leq n$, we see that the smallest and largest possible areas are $\frac{1}{2n^2(n-1)}$ and $\frac{1}{2n}$, respectively.  We further observe that these areas can only occur in the narrowest and widest strips. (The narrowest strips are $(\frac{1}{n},\frac{1}{n-1})$ and $(\frac{n-2}{n-1}, \frac{n-1}{n})$, while the widest are $(0,\frac{1}{n})$ and $(\frac{n-1}{n}, 1)$.) Combining this observation with Theorem \ref{thm:pidomain}, we obtain the following result.  The details are left to the reader.

\begin{corollary}
Among all $\pi \in \Sos_n$,
\[
 \frac{1}{2n^2(n-1)} \leq \ar(S(\pi)) \leq \frac{1}{2n}.
\]
The upper bound is attained by the identity permutation $012\cdots n$ and its reversal $n\cdots 210$. The lower bound is attained on the following set of four permutations:
\[
 \{0 n 1 2 \ldots (n-1), (n-1) \ldots 2 1 n 0, 1 2 \ldots  (n-1) 0 n, n 0 (n-1) \ldots 2 1\}.
\]
\end{corollary}

To finish, we remark that Observation \ref{lem:vert_coord} can also be used to draw a direct connection between the three gaps theorem and the three areas theorem.

First notice that on the interval $(a/b,c/d)$, each gap $\delta_k= \delta_k(\alpha)$ is a continuous function of $\alpha$, and these can be used to express the areas of the domains. To be precise, the area of $S(\pi\circ c^{n-k})$ (the $k$th domain from the top) is the following integral:
\begin{equation}\label{eq:integral}
 \int_{a/b}^{c/d} \delta_k(\alpha) \, d\alpha.
\end{equation}
By the three gaps theorem, we have
\[
 \delta_k(\alpha) = \begin{cases} \delta_0(\alpha) & \mbox{if } \pi(k)\leq n-b,\\
 \delta_0(\alpha)+\delta_n(\alpha) & \mbox{if } n-b < \pi(k) < d,\\
 \delta_n(\alpha) & \mbox{if } d\leq \pi(k).
 \end{cases}
\]
Thus, using the integrals in equation \eqref{eq:integral}, we can recover the three areas theorem.

In particular, we observe that if $\delta_k(\alpha) = \delta_0(\alpha)$, the domain is a triangular region with one side along the right edge of the vertical strip and if $\delta_k(\alpha) = \delta_n(\alpha)$, the domain is a triangular region with one side along the left edge of the vertical strip. Otherwise, if $\delta_k(\alpha) = \delta_0(\alpha)+\delta_n(\alpha)$, then the domain is a trapezoid.

From this perspective, Theorem \ref{thm:3probs} and Observation \ref{obs:distinct} yield the following corollary.

\begin{corollary}
In a vertical strip $(a/b,c/d)\times [0,1)$, there are trapezoidal domains if and only if $n+1 < b+d$.
\end{corollary}

While we have outlined many of the core features of S\'os permutations, there is much left undone. For example, O'Bryant \cite{Obryant} gives a formula for the sign of a S\'os permutation in terms of $\alpha$, and this could be extended to general $\alpha$ and $\beta$. Even better would be a formula for the number of inversions of $\pi_{\alpha,\beta}$ in terms of $\alpha$ and $\beta$. Boyd and Steele study longest increasing subsequences for a sequence of S\'os permutations with $\alpha$ fixed and $n\to \infty$. The distribution of longest increasing subsequences on the full symmetric group has connections to Plancharel measure on Young diagrams and eigenvalues of random matrices. What about this distribution restricted to the S\'os permutations?

In a different vein, one could envision other sets of permutations generated by functions on a cylinder. What about the ``quadratic'' permutations obtained by sorting the list $f(0), f(1), \ldots,f(n)$ with $f(x) = \alpha x^2 + \beta x + \gamma \mod 1$?

\end{document}